\documentclass[11pt]{amsart}

\usepackage[utf8]{inputenc}
\usepackage[T1]{fontenc}

\usepackage{amsmath,amsthm,amssymb}

\usepackage{graphics,graphicx}
\usepackage{tikz-cd}
\usetikzlibrary{calc}
\usetikzlibrary{arrows,patterns}

\usepackage{xfrac}

\usepackage[backref=page,colorlinks=true,allcolors=blue]{hyperref}

\newtheorem{thm}{Theorem}[section]
\newtheorem{cor}[thm]{Corollary}
\newtheorem{lem}[thm]{Lemma}
\newtheorem{prop}[thm]{Proposition}
\newtheorem{quest}[thm]{Question}

\theoremstyle{definition}

\DeclareMathOperator{\Aut}{Aut}
\DeclareMathOperator{\Stab}{Stab}
\DeclareMathOperator{\diag}{diag}

\DeclareMathOperator{\Ric}{Ric}

\DeclareMathOperator{\SL}{SL}
\DeclareMathOperator{\Sp}{Sp}
\DeclareMathOperator{\SO}{SO}
\DeclareMathOperator{\GO}{O}
\DeclareMathOperator{\GL}{GL}
\DeclareMathOperator{\PGL}{PGL}

\DeclareMathOperator{\barDH}{Bar}
\DeclareMathOperator{\Cone}{Cone}
\DeclareMathOperator{\Bl}{Bl}

\newcommand{\dint}{\mathrm{d}}

\newcommand\killing[2]{\left\{ #1, #2 \right\}}
\newcommand\blPPP[2]{X_{#1, #2}}  
\newcommand\blQQ[2]{Y_{#1, #2}}

\newcommand{\bbC}{\mathbb{C}}
\newcommand{\bbR}{\mathbb{R}}
\newcommand{\bbQ}{\mathbb{Q}}
\newcommand{\bbZ}{\mathbb{Z}}
\newcommand{\bbN}{\mathbb{N}}
\newcommand{\bbP}{\mathbb{P}}

\title{Examples of K-unstable Fano manifolds}
\author{Thibaut Delcroix}
\email{thibaut.delcroix@umontpellier.fr}
\urladdr{http://delcroix.perso.math.cnrs.fr/}
\date{2019}

\begin{document}

\begin{abstract}
We examine various examples of horosymmetric manifolds which exhibit interesting properties with respect to canonical metrics. 
In particular, we determine when the blow-up of a quadric along a linear subquadric admits Kähler-Einstein metrics, providing infinitely many examples of manifolds with no Kähler-Ricci solitons that are not K-semistable. 
Using a different construction, we provide an infinite family of Fano manifolds with no Kähler-Einstein metrics but which admit coupled Kähler-Einstein metrics. 
Finally, we elaborate on the relationship between Kähler-Ricci solitons and the more general concept of multiplier Hermitian structures and illustrate this with examples related to the two previous families.
\end{abstract}

\maketitle

\section{Introduction}

When does a Fano manifold admits a Kähler-Einstein metric? 
If it does not, how can we measure this failure?
What alternative canonical metric can we obtain on a Fano manifold?
The goal of this article is not to answer these questions but to give them more visibility and provide a detailed study of a few examples. 
By studying the blow-ups of projective spaces at two disjoint linear subspaces, we provide an infinite family of examples of manifolds with no Kähler-Einstein metrics but coupled Kähler-Einstein metrics. 
By studying the blow-ups of quadrics at a linear subquadric, we provide an infinite family of examples of manifolds that are not K-semistable and admit no Kähler-Ricci solitons. 

Let us be a bit more precise. Consider a Fano manifold $X$. 
A Kähler form (abusively, metric) $\omega\in c_1(X)$ is Kähler-Einstein if it satisfies the equation 
\[ \Ric(\omega)=\omega \] 
where $\Ric(\omega)$ denotes the Ricci form of $\omega$. 
We say a Fano manifold is Kähler-Einstein if it admits a Kähler-Einstein form. 
A convincing numerical measure of how a Fano manifold fails to be Kähler-Einstein is given by the \emph{greatest Ricci lower bound}:
\[ R(X)=\sup\left\{t\in[0,1] \mid \exists \omega\in c_1(X), \Ric(\omega)\geq t\omega \right\}. \]  
If $X$ is Kähler-Einstein then $R(X)=1$ while the converse direction does not hold. 
In K-stability terms, $R(X)=1$ is equivalent to K-semistability of $X$. 

Several candidate alternative metrics on Fano manifolds with no Kähler-Einstein metrics have been proposed over the years. 
The best known are the Kähler-Ricci solitons, which satisfy the equation 
\[ \Ric(\omega)-L_{\xi}\omega=\omega \] 
where $\xi$ is a holomorphic vector field and $L_{\xi}$ denotes the Lie derivative with respect to $\xi$. 
They actually belong to a large family of metrics introduced by Mabuchi as \emph{multiplier Hermitian structures}: $X$ is equipped with a multiplier Hermitian structure if there exists a Kähler form $\omega\in c_1(X)$, a holomorphic vector field $\xi$ and a smooth real valued concave function $h$ such that 
\begin{equation}
\label{eqn_mhs}
\Ric(\omega) - \sqrt{-1}\partial \bar{\partial} h\circ \theta = \omega 
\end{equation}
where $\theta$ is defined by $L_{\xi}\omega=\sqrt{-1}\partial \bar{\partial} \theta$. 
Kähler-Einstein metrics correspond to the case $g=0$ and Kähler-Ricci solitons to the case $g=\text{id}$. 
When $g$ is the logarithm of an affine function, such metrics are called Mabuchi metrics. 
The advantage of Mabuchi metrics over Kähler-Ricci solitons is that such metrics have a more algebraic nature. 
For example, the holomorphic vector field is essentially well determined by the manifold alone in the case of Kähler-Ricci solitons or of Mabuchi metrics. 
In the case of solitons, the vector field may very well generate a non-closed subgroup whereas this is impossible for Mabuchi metrics. 
This can make an important difference in the study of moduli spaces of Fano manifolds with canonical metrics. 
Motivated by this, we want to advertise the following open question. 
\begin{quest}
\label{ques_pol}
Is it true that a Fano manifold $X$ admits a Kähler-Ricci soliton if and only if there exists a multiplier Hermitian structure as defined by Equation~\eqref{eqn_mhs}, where $h$ is the logarithm of a polynomial?
\end{quest}

Another candidate for canonical metrics is given by coupled Kähler-Einstein metrics.
These were introduced by Hultgren and Witt Nyström in \cite{HWN17} and in the simplest case of pairs of metrics, may be defined as follows. 
A Fano manifold $X$ admits a pair of coupled Kähler-Einstein metrics if there exists a pair $(\omega_1,\omega_2)$ of Kähler forms such that 
\[ \Ric(\omega_1)=\Ric(\omega_2)=\omega_1+\omega_2. \]
The advantage here is that no holomorphic vector field is involved, and the drawback is that several Kähler classes have to be considered. 
Until the present paper, only two examples of Fano manifolds with no Kähler-Einstein metrics but a pair of coupled Kähler-Einstein metrics were known \cite{Hul,DH}. 

The first example of a Fano manifold $X$ with no Kähler-Ricci solitons and $R(X)<1$ was obtained by the author in \cite{DelTh,DelKE}.  
The example is a biequivariant compactification of the group $\Sp_4$, obtained by blowing up the unique closed orbit in the wonderful compactification of this group. 
The author later provided two additional examples in \cite{DelKSSV}, as compactifications of the group $\SO_4$, however the examples were still sporadic at that point. 
Such examples provide a useful illustration of the variety of situations regarding K-stability, and are the only known illustrations of type 2 singularity formation along the Kähler-Ricci flow on Fano manifolds (see \cite{LTZ}). 
It should be noted that Fujita obtained in \cite{Fuj17} examples of Fano manifolds of Picard number one which are K-unstable. To the author's knowledge, it is not known if these examples admit Kähler-Ricci solitons. 

The infinite family of examples presented here share a strong similarity with the previous examples: they are equivariant compactifications of symmetric spaces (actually one of the above mentioned compactifications of $\SO_4$ is part of this family). 
This allows us to use the K-stability criterions proved by the author in \cite{DelKSSV}, which reduces the problem to determining the sign of a finite number of integrals of polynomials on polytopes. 
While such integrals can be computed, it yields in general complicated expressions. 
In particular, the sign is not easily determined without an explicit computation, making the determination of K-stability on infinite families a challenging task. 

The family of examples we consider here is formed as we said by blow-ups of quadrics along lower dimensional linear subquadrics. 
An infinity of these are Kähler-Einstein, and an infinity are as stated (not K-semistable and with no Kähler-Ricci solitons). 
\begin{thm}
\label{thm_blQ}
Let $Q^n$ denote the quadric of dimension $n$, and $Q^k$ denote a linear subquadric of $Q^n$ of dimension $k$. 
Then if $2\leq k \leq n-3$, the blow-up $\Bl_{Q^k}(Q^n)$ of $Q^n$ along $Q^k$ is K-unstable and does not admit any Kähler-Ricci soliton. 
In the other cases, $\Bl_{Q^k}(Q^n)$ is Kähler-Einstein. 
\end{thm}
The lowest dimensional example is $\Bl_{Q^2}(Q^5)$. It has dimension five, so the following question remains open. 
\begin{quest}
Does there exists examples of Fano threefolds with no Kähler-Ricci solitons that are K-unstable? What about Fano fourfolds?
\end{quest}

In the process of introducing the symmetric structure on the family of manifolds that provide our examples, a natural detour is to start by considering the blow-ups of projective spaces along linear subspaces. 
The blow-up of a projective space along a linear subspace has a non-reductive automorphism group, satisfying Matsushima's obstruction \cite{Mat57}. 
They provided as such the first example of a Fano manifold with no Kähler-Einstein metrics: the blow-up of $\bbP^2$ at a point. 

Furthermore, the blow-up of $\bbP^4$ along a line and a disjoint plane provided the first example of a Fano manifold with reductive automorphism group and non-vanishing Futaki invariant as highlighted by Futaki himself \cite{Fut83}. 
As a toric manifold, $\Bl_{\bbP^1,\bbP^2}\bbP^4$ admits a Kähler-Ricci solitons: Wang and Zhu proved in \cite{WZ04} that any toric Fano manifold admits a Kähler-Ricci soliton. 
It actually admits better metrics in two directions: it admits Mabuchi metrics and coupled Kähler-Einstein metrics. 

More generally, the blow-up of a projective space along two complementary linear subspaces also has reductive automorphism group and non-vanishing Futaki invariant. 
We show that an infinity of these manifolds are manifolds with no Kähler-Einstein metrics but pairs of coupled Kähler-Einstein metrics. 
\begin{thm}
\label{thm_cKE_BlP}
For $k$ large enough, $\Bl_{\bbP^{k-1},\bbP^k}\bbP^{2k+1}$ admits pairs of coupled Kähler-Einstein metrics but no Kähler-Einstein metrics. 
\end{thm}
We believe this holds more generally and our proof could easily be adapted to prove other particular cases but we found no general proof. 
\begin{quest}
Does there exist a pair of coupled Kähler-Einstein metrics on all blow-ups of the projective space along complementary subspaces ? 
Is it true that there exist no coupled Kähler-Einstein metrics on $\Bl_{Q^k}Q^n$, for $2\leq k\leq n-3$?
\end{quest}
Similarly, we are convinced that Mabuchi metrics exist on all blow-ups of projective space we consider. Finding no elegant general proof of this fact, we instead provide a positive answer to Question~\ref{ques_pol} in a particular case. 
\begin{thm}
\label{thm_polMab}
Let $X$ be a homogeneous $\bbP^1$-bundle over a rational homogeneous manifold $G/P$, where $G$ is a reductive group and $P$ a parabolic subgroup of $G$. Assume $X$ is Fano, then $X$ admits Kähler-Ricci solitons, and $X$ admits multiplier Hermitian structures as in Equation~\eqref{eqn_mhs}, where $h$ is the logarithm of a polynomial of degree equal to the dimension of $G/P$. 
\end{thm}

In the body of the paper we will prove the results stated in this introduction. 
The results to prove are translated into combinatorial conditions by using \cite{DH}. 
More precisely, Theorem~\ref{thm_blQ} is proved using \cite[Corollary~1.3]{DH}, Theorem~\ref{thm_cKE_BlP} is proved using \cite[Corollary~1.5]{DH}, and Theorem~\ref{thm_polMab} is proved using \cite[Theorem~1.2]{DH}.  
As already mentioned, it does not mean that checking the conditions is trivial, as soon as we consider infinite families. 
Nonetheless the arguments to be used pertain to elementary mathematics. 
An additional goal of the paper is the description of all the combinatorial data associated to the families of examples considered (in particular a full description of the Picard group and moment polytopes associated to ample line bundles). 
Our hope is that these detailed examples will provide an easier entry point to the theory of horosymmetric and spherical manifolds, and can more readily be used to check properties of other canonical metrics or other geometric properties. 

In section~\ref{sec_BlP} we consider the blow-ups of projective spaces along complementary linear subspaces as horospherical manifolds (the combinatorial data can be used as well to study blow-ups at a single linear subspace as horospherical manifolds), and prove Theorem~\ref{thm_cKE_BlP} and Theorem~\ref{thm_polMab}. 
In sections~\ref{sec_BlQss}~and~\ref{sec_BlQred}, we study the blow-up of a quadric along a linear subquadric as a symmetric variety and prove Theorem~\ref{thm_blQ}. 
In section~\ref{sec_BlQss} we treat the case when the symmetric space is semisimple, which contains the K-unstable cases. 
In section~\ref{sec_BlQred} we treat the remaining case. 
This final section also contains an interesting family: the blow-up of a quadric at a point admits no Mabuchi metrics, starting from dimension three.

Each one of the three sections follows the same structure. 
We begin by a geometric description of the manifolds and of the orbits under the natural group action to consider.
We then describe the full combinatorial data associated to these manifolds, as horosymmetric varieties. 
First we describe the combinatorial data encoding the manifold itself, then we describe the combinatorial data associated to each ample $\bbR$-divisor. 
Finally we examine the combinatorial conditions encoding the existence of canonical metrics to prove our main results. 

\subsection*{Acknowledgements}
It is my pleasure to thank Jakob Hultgren for our joint work which allowed to prove half of the results in this article, and which provided motivation to study these particular families of examples. 
Part of this work was initially written in preparation for talks at two conferences during the summer of 2019. 
I thank Ivan Cheltsov and Simone Calamai for the invitations to these conferences.   
This research was accomplished in part during my postdoc in Strasbourg funded by the Labex IRMIA, and received partial funding from ANR Project FIBALGA ANR-18-CE40-0003-01. 

\section{Blow-ups of projective space along complementary linear subspaces}
\label{sec_BlP}

\subsection{Description of the manifolds}
Let $p$ and $n$ be two integers such that   
\[2\leq p\leq n-2, \] 
Consider the complex reductive group $G:=S(\GL_p\times \GL_{n-p})$. 
We consider $G$ as embedded in $\GL_n$ by block-diagonal matrices: if $(A,B)\in G$, we identify it with $\begin{pmatrix} A & 0 \\ 0 & B \end{pmatrix}$.  

Write $\bbC^n=\bbC^p\oplus \bbC^q$ and consider the induced action of the group $G$. 
There are four orbits under this action: the fixed point $\{0\}$, the pointed linear subspaces $(\bbC^p\setminus \{0\})\times \{0\}$ and $\{0\} \times (\bbC^q\setminus \{0\})$, and the open dense orbit $(\bbC^p\setminus \{0\}) \times (\bbC^q\setminus \{0\})$. 
It induces as well an action of $G$ on $\bbP^{n-1}$ with three corresponding $G$-orbits: an open dense orbit, and two disjoint linear subspaces (which we denote by $\bbP^{p-1}$ and $\bbP^{q-1}$ in an abuse of notations). 

The blow-ups of the projective space along one or both of these linear subspaces are Fano manifolds. 
We will focus on the manifold obtained by blowing up both orbits, 
and we will denote the manifold by 
\[ \blPPP{n}{p} := \Bl_{\bbP^{p-1},\bbP^{q-1}}(\bbP^{n-1}). \]
By the description above, this manifold is equipped with an action of $G$. 

It is a standard fact, though not obvious, that the neutral components of the automorphism group $\Aut(\Bl_Y(X))$ of a blow-up and of the stabilizer $\Stab_{\Aut(X)}(Y)$ of the blown-up submanifold in the automorphism group of the initial manifold are isomorphic. 
The non-obvious direction may be seen as a consequence of Blanchard's Lemma (see \cite{BSU13,Bla56}). 

The stabilizer of $\bbP^{p-1}$ in the automorphism group $\Aut(\bbP^{n-1})=\PGL_n$ of the projective space is the image of the group of upper block-triangular invertible matrices with two square diagonal blocks of size $p$ and $q$.    
The neutral component $\Aut^0(\Bl_{\bbP^{p-1}}(\bbP^{n-1}))$ of $\Aut(\Bl_{\bbP^{p-1}}(\bbP^{n-1}))$ is thus exactly the stabilizer described above. 
In particular, this group is not reductive and the image of $G$ in it is a Levi subgroup. 
For $\Bl_{\bbP^{p-1},\bbP^{q-1}}(\bbP^{n-1})$, the same argument shows that $\Aut^0(\Bl_{\bbP^{p-1},\bbP^{q-1}}(\bbP^{n-1}))$ is the Levi subgroup above, hence the image of $G$ is the full connected automorphism group. 
Note in particular that the automorphism group of $\blPPP{n}{p}$ is reductive. 

Under the action of $G$, the manifold $\blPPP{n}{p}$ is a horospherical manifold, as we will explain in the next section. 
This structure, taking into account the full group of automorphisms, is the most natural to consider. 
However, we should note that $\blPPP{n}{p}$ is:
\begin{itemize}
\item a toric manifold if we consider only the action of a maximal torus,
\item a cohomogeneity one manifold if we consider only the action of a maximal compact subgroup
\item a $\bbP^1$-bundle over the product of projective spaces $\bbP^{p-1}\times \bbP^{q-1}$. 
\end{itemize} 

\subsection{Combinatorial data associated to the manifold}

\subsubsection{The open orbit}
Let $T\subset \SL_n$ denote the maximal torus (both in $G$ and $\SL_n$) formed by diagonal matrices, and denote the roots of $\SL_n$ with respect to $T$ by 
\begin{align*}
\alpha_{i,j}: T & \to \bbC^* \\
t & \mapsto \frac{d_i}{d_j}
\end{align*}
where $t=\diag(d_1,\ldots,d_n)$, for $1\leq i\neq j\leq n$.
The roots of $G$ are the $\alpha_{i,j}$ with $1\leq i\neq j \leq p$ or $p+1\leq i\neq j\leq n$.
The set of characters of $T$ is denoted by $\mathfrak{X}(T)$ and the set of one-parameter subgroups is denoted by $\mathfrak{Y}(T)$. The Killing form of $G$ is denoted by $\killing{x}{y}$ for $x,y\in\mathfrak{g}$.

It is immediate to determine the stabilizer $H$ of the point $(1:0\cdots :0:1:0:\cdots :0)$ in $\bbP^{n-1}$, which is in the open orbit (the $1$ are at position $1$ and $(p+1)$). 
Let $P$ denote the parabolic subgroup of $G$ containing $T$ such that the roots of its unipotent radical are 
\[ \Phi_{P^u}=\{\alpha_{1,2}, \ldots, \alpha_{1,p},\alpha_{p+1,p+2}, \ldots, \alpha_{p+1,n}\}. \] 
Then $H$ is the kernel in $P$ of the character of $P$ defined by $\alpha_{1,p+1}$.

As a consequence, $H$ contains the unipotent radical of $P$, and $G/H$ is a rank one horospherical homogeneous space. 
To be consistent with the conventions in \cite{DH} and \cite{DelHoro}, 
we fix an involution $\sigma$ of $\mathfrak{t}$ such that $\mathfrak{t}^{\sigma}=\ker(\alpha_{1,p+1})$. 
It amounts to choosing the fixed point set of $\sigma$, a complement to $\bbR\alpha_{1,p+1}$ in $\mathfrak{X}(T)\otimes \bbR$. 
Here, we choose the orthogonal of $\alpha_{1,p+1}$ with respect to the Killing form to define $\sigma$. 

Let us describe the combinatorial invariants corresponding to the spherical homogeneous space $G/H$. 
We do this with respect to the Borel subgroup $B$ of $G$ containing $T$ whose corresponding simple roots are 
\[ S = \{\alpha_{2,1},\ldots,\alpha_{p,p-1},\alpha_{p+2,p+1},\ldots \alpha_{n,n-1}\}, \] 
to ensure again consistency with the conventions in \cite{DH} 
and \cite{DelHoro}.
The \emph{spherical lattice} $\mathcal{M}$ is then $\mathfrak{X}(T/T\cap H)=\bbZ \alpha_{1,p+1}$. 
The \emph{valuation cone} is the vector space generated by $\alpha_{1,p+1}$: $\mathcal{V}= \mathcal{M}\otimes \bbR$. 
The \emph{colors} are in bijection with the set $-\Phi_{P^u}\cap S=\{\alpha_{2,1},\alpha_{p+2,p+1}\}$, we denote these by $D_{\alpha_{2,1}}$ and $D_{\alpha_{p+2,p+1}}$ respectively. 
Recall that given a root $\alpha\in \mathfrak{X}(T)$ of $G$, the coroot $\alpha^{\vee}\in \mathfrak{Y}(T)$ is defined by $\alpha(x)=2 \frac{\killing{x}{\alpha^{\vee}}}{\killing{\alpha^{\vee}}{\alpha^{\vee}}}$ for all $x\in \mathfrak{Y}(T)\otimes\bbR$.
The \emph{color map} sends $D_{\alpha}$ to the restriction $\alpha^{\vee}|_{\mathcal{M}\otimes\bbR}$ for $\alpha\in\{\alpha_{2,1},\alpha_{p+2,p+1}\}$. 
The images are thus respectively $-e$ and $e$ where $e=\alpha_{1,p+1}^*$ denotes the element dual to $\alpha_{1,p+1}$ in $(\mathcal{M}\otimes \bbR)^*$, as illustrated in Figure~\ref{fig_MtimesR}. 

\begin{figure}
\centering
\caption{$\mathcal{M}\otimes\bbR$}
\label{fig_MtimesR}
\begin{tikzpicture}  
\draw (-5,0) -- (6,0);
\draw (0,0) node{$\bullet$};
\draw (0,0) node[below]{$0$};
\draw (2.5,0) node{$\bullet$};
\draw (2.5,0) node[below right]{$e=\alpha_{1,p+1}^*=\alpha_{p+2,p+1}^{\vee}|_{\mathcal{M}\otimes\bbR}$};
\draw (-2.5,0) node{$\bullet$};
\draw (-2.5,0) node[below left]{$-e=\alpha_{2,1}^{\vee}|_{\mathcal{M}\otimes\bbR}$};
\end{tikzpicture}
\end{figure}

Note that if we considered the case $p=1$ or $n-p=1$, only one of the two roots $\alpha_{2,1}$ or $\alpha_{p+2,p+1}$ would be a well defined roots of $G$. There would correspondingly be only one projective subspace giving rise to a non-trivial blow-up. We will not give the details but the combinatorial description is a minor variation of the one presented here. 

\subsubsection{Colored fans and embeddings}
The colored fan of $\bbP^{n-1}$ is the fully colored fan 
\[ \{(\bbR_+ e,D_{\alpha_{p+2,p+1}}),(-\bbR_+ e,D_{\alpha_{2,1}}),(\{0\},\emptyset)\}, \]
where the first cone corresponds to the orbit $\bbP^{p-1}$, the second to $\bbP^{q-1}$ and the last to the open orbit. 
The colored fan of $\Bl_{\bbP^{q-1}}\bbP^{n-1}$ is 
\[ \{(\bbR_+ e,D_{\alpha_{p+2,p+1}}),(-\bbR_+ e,\emptyset),(\{0\},\emptyset)\}. \]  
The colored fan of $\Bl_{\bbP^{p-1}}\bbP^{n-1}$ is 
\[ \{(\bbR_+ e,\emptyset),(-\bbR_+ e,D_{\alpha_{2,1}}),(\{0\},\emptyset)\}. \] 
Finally, the colored fan of $\Bl_{\bbP^{p-1},\bbP^{q-1}}\bbP^{n-1}$ is the colorless fan 
\[ \{(\bbR_+ e,\emptyset),(-\bbR_+ e,\emptyset),(\{0\},\emptyset)\}. \] 

\subsection{Combinatorial data associated to divisors}

\subsubsection{Picard group, isotropy characters and polytopes:}
Recall that the Picard group of a general spherical variety is described by Brion in \cite{Bri89}. 
We recalled this description and introduced a few convenient notions in \cite{DelHoro,DH} to deal with moment polytopes in the special case of horosymmetric manifolds. We will use the language of \cite{DelHoro} here, notably we will use the terms \emph{special divisor} and \emph{isotropy character} as defined in that paper. 

The Picard group of $\Bl_{\bbP^{p-1},\bbP^{q-1}}\bbP^{n-1}$ is of rank $3$. 
Let $D_{\pm}$ denote the $G$-invariant divisor associated to the ray $\pm \bbR_+ e$.
The Picard group is generated by $D_{\alpha_{2,1}}$, $D_{\alpha_{p+2,p+1}}$, $D_+$ and $D_-$ with the relation 
\[ D_++D_{\alpha_{p+2,p+1}}=D_-+D_{\alpha_{2,1}}. \] 
For $D_{\pm}$, they already are \emph{special} divisors, and the isotropy characters are trivial. 
Up to renormalizing the action by a character of $G$, the linearized line bundle associated to $D_{\alpha_{2,1}}$ admits a $B$-semi-invariant section whose $B$-weight $\varpi_{2,1}$ is given by 
$\mathrm{diag}(d_1,\cdots,d_n)\mapsto d_1^{-1}$ 
(this is, up to a character of the reductive group $G$ the fundamental weight of $\alpha_{2,1}$). 
Then the linearly equivalent \emph{special} $\bbQ$-divisor is 
\begin{align*} D_{\alpha_{2,1}}-\frac{\killing{\varpi_{2,1}}{\alpha_{1,p+1}}}{\killing{\alpha_{1,p+1}}{\alpha_{1,p+1}}}(D_++D_{\alpha_{p+2,p+1}}-D_--D_{\alpha_{2,1}}) & \\
  =  \frac{1}{2}(D_++D_{\alpha_{p+2,p+1}}+D_{\alpha_{2,1}}-D_-)  &
\end{align*}
and its isotropy character is 
\[ \chi_{2,1}:=\varpi_{2,1}-\frac{\killing{\varpi_{2,1}}{\alpha_{1,p+1}}}{\killing{\alpha_{1,p+1}}{\alpha_{1,p+1}}}\alpha_{1,p+1}
=\varpi_{2,1}+\frac{1}{2}\alpha_{1,p+1}. \]

As a consequence, special $\bbR$-divisors are all the  
\[ D(a_c,a_+,a_-):=a_c(D_{\alpha_{2,1}}+D_{\alpha_{p+2,p+1}})+a_+D_+ + a_-D_- \] 
for $a_c$, $a_+$, $a_-\in\bbR$. 
The isotropy character of such a divisor is $2a_c\chi_{2,1}$, and the associated \emph{special polytope} is defined by 
\[ \Delta(a_c,a_+,a_-)=\{ t\alpha_{1,p+1} \mid \max(-a_+,-a_c) \leq t\leq \min(a_-,a_c) \}. \]
The moment polytope, provided the divisor is ample, is
\[ \Delta^+(a_c,a_+,a_-) = 2a_c\chi_{2,1}+\Delta(a_c,a_+,a_-). \]
Note that Brion's ampleness criterion \cite[Théorème~3.3]{Bri89} states, in our case, that $D(a_c,a_+,a_-)$ is ample if and only if $a_+<a_c$, $a_-<a_c$ and $a_++a_->0$. 

\begin{figure}
\centering
\caption{Moment polytope for an ample divisor on $\Bl_{\bbP^{p-1},\bbP^{q-1}}\bbP^{n-1}$}
\label{fig_polBlP}
\begin{tikzpicture}  
\draw [dashed] (0,-.5) -- (0,3.5);
\draw [dashed] (-1.5,0) -- (3.5,0);
\draw [dotted] (-1,0) -- (-1,3.5);
\draw [dotted] (1.5,0) -- (1.5,3.5);
\draw [thick] (-1,3.5) -- (1.5,3.5);
\draw (0,0) node{$\bullet$};
\draw (-1,0) node{$\bullet$};
\draw (-1,0) node[below left]{\small $-a_+\alpha_{1,p+1}$};
\draw (1.5,0) node{$\bullet$};
\draw (1.5,0) node[below right]{\small $a_-\alpha_{1,p+1}$};
\draw (-1,3.5) node{$\bullet$};
\draw (1.5,3.5) node{$\bullet$};
\draw (1.5,3.5) node[right]{$\Delta^+(a_c,a_+,a_-)$};
\draw (0,3.5) node{$\bullet$};
\draw (0,3.5) node[above right]{\small $2a_c\chi_{2,1}$};
\end{tikzpicture}
\end{figure}
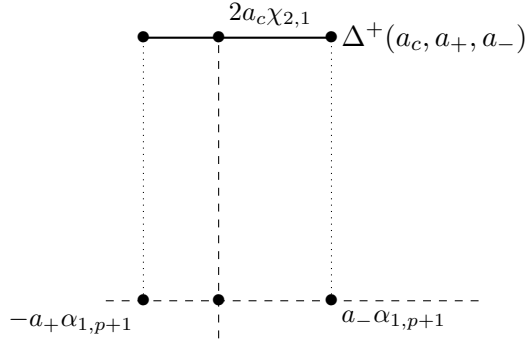

The special divisor $D\left(\frac{n}{2},p+1-\frac{n}{2},\frac{n}{2}-p+1\right)$ represents the anticanonical divisor, its isotropy character is $n\chi_{2,1}$, thus the corresponding moment polytope is  
\[ \Delta^+\left(\frac{n}{2},p+1-\frac{n}{2},\frac{n}{2}-p+1\right) = n\chi_{2,1}+\Delta\left(\frac{n}{2},p+1-\frac{n}{2},\frac{n}{2}-p+1\right). \] 

\subsubsection{Duistermaat-Heckman polynomial and barycenters}

Key ingredients in the combinatorial criterions for existence of canonical metrics are the Duistermaat-Heckman polynomial and barycenters of moment polytopes with respect to this polynomial. We give here a first expression of these in our special case. 

It is enough to express the Duistermaat-Heckman polynomial $P_{DH}$ in the plane generated by $\alpha_{1,p+1}$ and $2\chi_{2,1}$ since every moment polytope lives in this plane. 
We have 
\begin{align*} 
P_{DH}(x\alpha_{1,p+1}+2y\chi_{2,1}) & = \prod_{\alpha\in -\Phi_{P^u}}\killing{\alpha}{x\alpha_{1,p+1}+2y\chi_{2,1}} \\
& = (-x+y)^{p-1}(x+y)^{n-p-1}
\end{align*}
The Duistermaat-Heckman barycenter corresponding to the ample divisor $D(a_c,a_+,a_-)$ is, for any choice of Lebesgue measure $\dint p$,
\begin{align*} 
\barDH(a_c,a_+,a_-) & = \frac{\int_{\Delta^+(a_c,a_+,a_-)} pP_{DH}(p)\dint p}{\int_{\Delta^+(a_c,a_+,a_-)} P_{DH}(p)\dint p}  \\
 & = 2a_c\chi_{2,1} + \left( \frac{\int_{-a_+}^{a_-} t(t+a_c)^{n-p-1}(a_c-t)^{p-1} \dint t}{\int_{-a_+}^{a_-} (t+a_c)^{n-p-1}(a_c-t)^{p-1}\dint t} \right)\alpha_{1,p+1}
\end{align*}

\subsection{Kähler-Einstein criterion}

Since $X$ is a toroidal horospherical manifold, the combinatorial criterion ruling the existence of Kähler-Einstein metrics on $X$ was originally proved by Podesta and Spiro \cite{PS10}. 
It was reproved on several occasions by the author, in \cite{DelKSSV} and \cite{DelHoro}. 
The criterion states that $X$ is Kähler-Einstein if and only if the barycenter of the moment polytope associated to the anticanonical divisor with respect to the Duistermaat-Heckman polynomial $P_{DH}$ coincides with the opposite of the sum of positive roots of the parabolic subgroup $P$. In formulas:
\begin{prop}[application of Theorem~B and Proposition~4.7 in \cite{PS10}]
The manifold $\Bl_{\bbP^{p-1},\bbP^{q-1}}\bbP^{n-1}$ is Kähler-Einstein if and only if 
\[ \barDH\left(\frac{n}{2},p+1-\frac{n}{2},\frac{n}{2}-p+1\right)= \sum_{\alpha\in -\Phi_{P^u}}\alpha. \]
\end{prop}

Note that 
\begin{align*}
\sum_{\alpha\in -\Phi_{P^u}}\alpha & = \alpha_{2,1}+\cdots+\alpha_{p,1}+\alpha_{p+2,p+1}+\cdots+\alpha_{n,p+1} \\
& = n\chi_{2,1} + \left(\frac{n}{2}-p\right)\alpha_{1,p+1} 
\end{align*}

Hence the manifold $\Bl_{\bbP^{p-1},\bbP^{q-1}}\bbP^{n-1}$ is Kähler-Einstein if and only if 
\[ I(n,p):=\int_{-(p+1-\frac{n}{2})}^{\frac{n}{2}-p+1} \left(t+p-\frac{n}{2}\right)\left(t+\frac{n}{2}\right)^{n-p-1}\left(\frac{n}{2}-t\right)^{p-1}\dint t = 0. \]

\begin{lem}
\label{thm_BlPP}
The manifold $\Bl_{\bbP^{p-1},\bbP^{q-1}}\bbP^{n-1}$ is Kähler-Einstein if and only if $n$ is even and $p=\frac{n}{2}$. Else we have $I(n,p) > 0$ if $p < \frac{n}{2}$ and $I(n,p) < 0$ if $p > \frac{n}{2}$.
\end{lem}

\begin{proof}
By using the variable $x=t+p-\frac{n}{2}$, we get 
\[ I(n,p)=\int_{-1}^1 x(x+n-p)^{n-p-1}(p-x)^{p-1} \dint x. \]
The integrand has a simple primitive, providing the explicit formula  
\begin{align*} 
I(n,p) & =\Big[ \frac{-1}{n}(p-x)^p(n-p+x)^{n-p} \Big]_{-1}^1 \\
& = \frac{-1}{n}\left( (p-1)^p(n-p+1)^{n-p} - (p+1)^p(n-p-1)^{n-p} \right) 
\end{align*}
We claim that the function $f:x\mapsto \left(\frac{x-1}{x+1}\right)^x$ is strictly increasing on $[1,\infty[$. Once the claim is proved, the theorem is as well in view of the last expression of $I(n,p)$.  
A direct computation of the derivative yields 
\[ f'(x)=\left(\ln\left(\frac{x-1}{x+1}\right)+\frac{2x}{x^2-1}\right)\left(\frac{x-1}{x+1}\right)^x. \]
By a power series expansion, for $0<y<1$ we have
\[ \ln\left(1-y\right)+\frac{y(2-y)}{2}\frac{1}{1-y} = 
\sum_{k\in \bbN} \frac{k+1}{2(k+3)}y^{k+3} >0. \]
Setting $y=\frac{2}{x+1}$, we obtain $f'(x)>0$ for $x>1$, thus the claim.
\end{proof}

\subsection{Mabuchi-type metrics}

The combinatorial criterion, proved by the author and Jakob Hultgren \cite[Corollary~1.5]{DH}, ruling the existence of Mabuchi metrics on $\Bl_{\bbP^{p-1},\bbP^{q-1}}\bbP^{n-1}$ reads as follows. 
We place ourselves in the situation where the manifold does not admit any Kähler-Eintein metrics, that is $p\neq \frac{n}{2}$. 
Set 
\[ J:=\int_{-1}^1 x^2(x+n-p)^{n-p-1}(p-x)^{p-1} \dint x>0. \]
Choose $A$ and $B\in \bbR$ such that 
\[ \int_{-1}^1 x(Ax+B)(x+n-p)^{n-p-1}(p-x)^{p-1} \dint x = AJ+BI = 0. \]  
There does not exist a Mabuchi metric if and only if for all such choices, the affine function $x\mapsto Ax+B$ vanishes on $[-1,1]$, that is, if and only if $-\frac{B}{A}=\frac{J}{I}\in [-1,1]$. 
If $p<\frac{n}{2}$ then $I>0$ and the latter condition is equivalent to $J-I\leq 0$. If $p>\frac{n}{2}$ then it is equivalent to $J+I\leq 0$. 

From checking on examples, it appears that the manifolds $\Bl_{\bbP^{p-1},\bbP^{q-1}}\bbP^{n-1}$ admit Mabuchi metrics in all cases. 
However we did not find an elegant way to prove this. 
On the other hand, we can easily show that there always exists a Mabuchi-type metric (or multiplier Hermitian structure) on any rank one toroidal horospherical Fano manifold. 

Recall that the notion of multiplier Hermitian structure was introduced by Mabuchi in \cite{Mab03} and consists of a Kähler metric $\omega\in c_1(X)$ solving the complex Monge-Ampère equation 
\[ \Ric(\omega)-\sqrt{-1}\partial\bar{\partial}h(\theta) = \omega \]
for some smooth real valued concave function $h$ defined on the range of $\theta$, where $\theta$ is the Hamiltonian function of $\omega$ with respect to some holomorphic vector field (equivalently, some action of $\bbC^*$). 

In the remaining of this subsection, we consider a Fano manifold $X$, which is an arbitrary rank one toroidal horospherical manifold under the action of a reductive group $G$ with a one-dimensional center containing a $\bbC^*$ subgroup. 
Equivalently, the manifold $X$ is a homogeneous $\bbP^1$-bundle over a rational homogeneous space $G/P$. 
The action of the central subgroup $\bbC^*$ is the natural action on the fibers.

\begin{proof}[Proof of Theorem~\ref{thm_polMab}]
Let $G/H$ be the open orbit in $X$. 
Let $P$ denote the normalizer of $H$ in $G$ and choose a Borel subgroup $B$ opposite to $P$, and $T$ a maximal torus in $B\cap P$. Then the moment polytope $\Delta^+$ of $X$ is of the form 
\[ \Delta^+=\left\{ t\varpi +\sum_{\alpha\in-\Phi_{P^u}}\alpha \mid -1\leq t\leq 1 \right\} \]
where $\varpi$ is a generator of the one dimensional lattice $\mathcal{M}$ (see \cite[p.~207]{Pas08}).
Then by \cite[Theorem~1.2]{DH}, there exists a multiplier Hermitian structure on $X$ if and only if we can find a smooth, real valued, concave function $h$ on $[-1,1]$ such that 
\[ \int_{-1}^1te^{h(t)}P_{DH}\left(t\varpi+\sum_{\alpha\in -\Phi_{P^u}}\alpha\right)\dint t = 0 \]
where, as usual, $P_{DH}$ denotes the Duistermaat-Heckman polynomial for $X$: $P_{DH}(x)=\prod_{\alpha\in -\Phi_{P^u}}\killing{\alpha}{x} $ for $x\in \mathfrak{X}(T)$.
Thanks to the symmetry of the polytope, an obvious choice is given by 
\[ e^{h(t)}=P_{DH}\left( -t\varpi+\sum_{\alpha\in -\Phi_{P^u}}\alpha \right). \]
The corresponding function $h$ is obviously smooth, real valued, and concave: $\killing{\alpha}{-t\varpi+\sum_{\alpha\in -\Phi_{P^u}}\alpha}>0$ for all $-1\leq t\leq 1$ and $\alpha\in -\Phi_{P^u}$ since the moment polytopes are strictly contained in the positive Weyl chamber for toroidal horospherical manifolds. 
\end{proof}

\subsection{Coupled Kähler-Einstein metrics}
In this section, we prove that there exists infinitely many couples $(n,p)$ such that the manifold $\blPPP{n}{p}$ admits no Kähler-Einstein metrics but admits a couple of coupled Kähler-Einstein metrics, by proving Theorem~\ref{thm_cKE_BlP} as stated in the introduction.

\begin{proof}[Proof of Theorem~\ref{thm_cKE_BlP}]
We place ourselves in the situation where $n=2k+1$ is odd, and $p=k$.

Using Theorem~1.2 of \cite{DH}, we want to find $a_c$, $a_+$ and $a_-$ such that 
\begin{itemize}
\item $\barDH(a_c,a_+,a_-)+\barDH(k+\frac{1}{2}-a_c,\frac{1}{2}-a_+,\frac{3}{2}-a_-)=n\chi_{2,1}+\frac{1}{2}\alpha_{1,k+1}$
\item $a_+<a_c<k+a_+$
\item $a_-<a_c<k-1+a_-$
\item $0<a_++a_-<2$.
\end{itemize}
Set $a_c'=k+\frac{1}{2}-a_c$, $a_+'=\frac{1}{2}-a_+$ and $a_-'=\frac{3}{2}-a_-$.
The three last conditions ensure that the divisors $D(a_c,a_+,a_-)$ and $D(a_c',a_+',a_-')$ are ample. Their sum is obviously the anticanonical divisor, and the first condition means that the coupled Kähler-Einstein equation is satisfied for this decomposition. 

Let $C(a_c,a_+,a_-)$ denote the quantity 
\[ 
\frac{\int_{-a_+}^{a_-} t(t+a_c)^{k}(a_c-t)^{k-1}\dint t}
{\int_{-a_+}^{a_-} (t+a_c)^{k}(a_c-t)^{k-1}\dint t}
+
\frac{\int_{-a_+'}^{a_-'} t(t+a_c')^{k}(a_c'-t)^{k-1}\dint t}
{\int_{-a_+'}^{a_-'} (t+a_c')^{k}(a_c'-t)^{k-1}\dint t}
 - \frac{1}{2}
\]
then the first condition is equivalent to $C(a_c,a_+,a_-)=0$. 
Note that the Kähler cone is connected and that $(a_c,a_+,a_-)\mapsto C(a_c,a_+,a_-)$ is continuous. Since $2p=2k<2k+1=n$, Theorem~\ref{thm_BlPP} implies that $C(\frac{n}{4},\frac{1}{4},\frac{3}{4})>0$, hence it is enough to find a triple $(a_c,a_+,a_-)$ (satisfying the ampleness conditions) with $C(a_c,a_+,a_-)<0$ to prove the existence of coupled Kähler-Einstein metrics. 

Set $a_c=k-2$, $a_+=\frac{1}{2}$ and $a_-=0$. 
Then we have $a_c=\frac{5}{2}$, $a_+=0$ and $a_-=\frac{3}{2}$.
Such a data satisfies the ampleness conditions. 
We will show $C(a_c,a_+,a_-)<0$ for $k$ large enough by showing that the first quotient of integrals in $C(a_c,a_+,a_-)$ converges to $-\frac{1}{4}$ and that the second quotient of integrals converges to $0$. 
In fact, if $a_c$ is linear in $k$, $a_+$ and $a_-$ are constant, then the quotient of integrals converges to the middle of $[-a_+,a_-]$, and if $a_+=0$, $a'_c$ and $a'_-$ are constant, then the quotient of integrals converges to $0$. 

Using the equation 
\[ \frac{(t+k-2)^{k}(k-2-t)^{k-1}}{(k-1)^{2k-1}} =\left(1+\frac{t-1}{k-1}\right)\left(1+\frac{t-1}{k-1}\right)^{k-1}\left(1-\frac{t+1}{k-1}\right)^{k-1} \]
we see that, uniformly in $t\in [-\frac{1}{2},0]$,
\[ \lim_{k\rightarrow \infty} \frac{(t+k-2)^{k}(k-2-t)^{k-1}}{(k-1)^{2k-1}} = e^{-2} \]
hence 
\[ \lim_{k\rightarrow \infty} \frac{\int_{-\frac{1}{2}}^{0} t(t+k-2)^{k}(k-2-t)^{k-1}\dint t}
{\int_{-\frac{1}{2}}^{0} (t+k-2)^{k}(k-2-t)^{k-1}\dint t} = -\frac{1}{4}. \]

The other limit is a bit less immediate. 
Fix $0<\epsilon<\frac{3}{2}$ and consider 
\begin{equation} \label{eqn_limit}
\int_{0}^{\frac{3}{2}} (t-\epsilon)\left(t+\frac{5}{2}\right)^{k}\left(\frac{5}{2}-t\right)^{k-1}\dint t. 
\end{equation}
If we prove that the latter integral is negative for $k$ large enough, then 
\[ 0 < \frac{\int_{0}^{\frac{3}{2}} t\left(t+\frac{5}{2}\right)^{k}\left(\frac{5}{2}-t\right)^{k-1}\dint t}{\int_{0}^{\frac{3}{2}} \left(t+\frac{5}{2}\right)^{k}\left(\frac{5}{2}-t\right)^{k-1}\dint t} <\epsilon \]
for $k$ large enough. 
Actually, proving this for just one $\epsilon<\frac{3}{4}$ is enough to get the conclusion. 
In the integral~\eqref{eqn_limit}, the negative contribution is the integral on $[0,\epsilon]$, while the positive contribution is the integral on $[\epsilon,\frac{3}{2}]$.

For $\epsilon\leq t \leq \frac{3}{2}$, we have 
\begin{align*}
\left(t+\frac{5}{2}\right)^{k}\left(\frac{5}{2}-t\right)^{k-1} & = \left(\frac{5}{2}\right)^{2k-1}\left(1+\frac{2t}{5}\right)\left(1-\left(\frac{2t}{5}\right)^2\right)^{k-1} \\
& \leq \left(\frac{5}{2}\right)^{2(k-1)}\left(1+\frac{3}{5}\right)\left(1-\left(\frac{2\epsilon}{5}\right)^2\right)^{k-1}
\end{align*}
hence 
\[ \int_{\epsilon}^{\frac{3}{2}} (t-\epsilon)\left(t+\frac{5}{2}\right)^{k}\left(\frac{5}{2}-t\right)^{k-1}\dint t \leq \left(\frac{9}{8}-\frac{3\epsilon}{2}+\frac{\epsilon^2}{2}\right)\left(\frac{5}{2}\right)^{2(k-1)}\left(1+\frac{3}{5}\right)\left(1-\left(\frac{2\epsilon}{5}\right)^2\right)^{k-1} \]

For the negative contribution, it suffices to restrict to $\left[0,\frac{1}{2(k-1)}\right]$. 
Set $t=\frac{s}{k-1}$ where $s\in \left[0,\frac{1}{2}\right]$. 
Then 
\[ \frac{(t+\frac{5}{2})^{k}(\frac{5}{2}-t)^{k-1}}{(\frac{5}{2})^{2k-1}} = \left(1+\frac{2s}{5(k-1)}\right)\left(1+\frac{2s}{5(k-1)}\right)^{k-1}\left(1-\frac{2s}{5(k-1)}\right)^{k-1} \]
Hence, uniformly in $s\in [0,\frac{1}{2}]$ (or in $t\in \left[0,\frac{1}{2(k-1)}\right]$), 
\[ \lim_{k\to \infty} \frac{(t+\frac{5}{2})^{k}(\frac{5}{2}-t)^{k-1}}{(\frac{5}{2})^{2k-1}} = 1. \]
We deduce that 
\[ \int_0^{\frac{1}{2(k-1)}} (t-\epsilon)\left(t+\frac{5}{2}\right)^{k}\left(\frac{5}{2}-t\right)^{k-1}\dint t \sim_{k\rightarrow \infty} \frac{-\epsilon\left(\frac{5}{2}\right)^{2k-1}}{2(k-1)} \]

Since 
\[ 
\left(\frac{9}{8}-\frac{3\epsilon}{2}+\frac{\epsilon^2}{2}\right)\left(\frac{5}{2}\right)^{2(k-1)}\left(1+\frac{3}{5}\right)\left(1-\left(\frac{2\epsilon}{5}\right)^2\right)^{k-1}= o\left(\frac{-\epsilon(\frac{5}{2})^{2k-1}}{2(k-1)}\right), \]
the negative contribution in integral~\eqref{eqn_limit} is dominant over the positive contribution. 
We deduce that integral~\eqref{eqn_limit} is negative for $k$ large enough.
\end{proof}

\section{Unstable blow-ups of quadrics along linear subquadrics}
\label{sec_BlQss}

\subsection{Description of the manifolds}
Choose as before two positive integers $p$ and $q$ such that $p+q=n$.
For this first section on blow-ups of quadrics, we assume $3\leq p\leq n-3$. 
We will deal with the case $p=n-2$ in the next section. 

Consider the standard quadratic form $\sum_{j=1}^{n}z_j^2$ on $\bbC^{n}$. 
Let $Q^{n-2}$ denote the associated $n-2$-dimensional projective quadric. 
It is equipped with an action of $G:=\SO_p\times \SO_q\subset \SO_n$ with four orbits, easily described in terms of the rational projections from $\bbP^{n-1}$ to $\bbP^{p-1}$ and $\bbP^{q-1}$. 
There is an open dense orbit, formed by the points at which both projections are defined and no projection lie in the corresponding quadrics, the codimension one orbit formed by the points at which both projections are defined and lie in the corresponding quadrics, and two closed orbits formed by the points where one of the projections is undefined. 

We will be interested in the blow-up $\blQQ{n}{p}$ of $Q^{n-2}$ along one of the closed orbits, say the quadric $Q^{p-2}$ in $\bbP^{p-1}$. 

The group of automorphisms of $Q^{n-2}$ is up to isogeny the group $\SO_n$. 
The connected group of automorphisms of $X$ is thus isogenous to the neutral component of the stabilizer of the quadric in $\bbP^{p-1}$ inside $\SO_n$. 
It is isogenous to $\SO_{p}\times \SO_{q}$. 

\subsection{Combinatorial data associated to the manifold}

\subsubsection{The open orbit}
The quadric and its blow-up, equipped with the action of $G$, are symmetric varieties. 
The open dense orbit is the symmetric space 
$\SO_p\times \SO_q / H$ where $H$ is the stabilizer of say $[1:0\cdots :0:1:0\cdots:0]$. 
It is the unique proper subgroup sandwiched between 
$\{1\}\times \SO_{p-1}\times \{1\} \times \SO_{q-1}$ 
and $S(\GO_{1}\times \GO_{p-1})\times S(\GO_1\times \GO_{q-1})$. 
In particular, it is not, properly speaking a product of symmetric spaces even though its restricted root system does decompose as $A_1\times A_1$. 
Let $\alpha$ and $\beta$ denote the simple restricted roots corresponding the first and second factor of $G$. 
We denote by $\alpha^{\vee}$ and $\beta^{\vee}$ the corresponding simple restricted coroots.
Note that the multiplicity of the restricted root $\alpha$ is $p-2$ and the multiplicity of $\beta$ is $q-2$.

Given the restricted root system, it is easy to determine the combinatorial data associated to the spherical homogeneous space $G/H$.  
The \emph{spherical lattice} $\mathcal{M}$ is the unique proper lattice sandwiched between the restricted weight lattice and the restricted root lattice for the restricted root system of type $A_1\times A_1$. 
It is thus generated by $\alpha$ and $\frac{1}{2}(\alpha+\beta)$. 
The \emph{valuation cone $\mathcal{V}$} is the negative restricted Weyl chamber. 
The \emph{colors} are in bijection with simple restricted roots and the image of the \emph{color map} is the set of simple restricted coroots.  
The description of colors in the last sentence is not immediate as it is not immediate that the color map is injective. 
However, it follows for example from the description of the Picard group which will follow in the next section, as the quadric is a Picard rank one manifold.  

\subsubsection{Colored fans and embeddings}
The colored fan of the quadric consists of the following colored cones:
\begin{itemize}
\item the point $(\{0\},\emptyset)$, for the open orbit,
\item the ray $(-\bbR_+ (\alpha^{\vee}+\beta^{\vee}),\emptyset)$ for the codimension one orbit,
\item $(\Cone(\alpha^{\vee},-(\alpha^{\vee}+\beta^{\vee})),\{\alpha^{\vee}\})$ for the subquadric $\text{Q}^{q-2}$,
\item and $(\Cone(\beta^{\vee},-(\alpha^{\vee}+\beta^{\vee})),\{\beta^{\vee}\})$ for the subquadric $\text{Q}^{p-2}$.
\end{itemize}

To obtain the colored fan of the blow up $\blQQ{n}{p}$ of the quadric along $\text{Q}^{p-2}$, one needs only replace the last cone by $(\Cone(-\alpha^{\vee},-(\alpha^{\vee}+\beta^{\vee})),\emptyset)$ and add the ray $(-\bbR_+ \alpha^{\vee}, \emptyset)$ (see Figure~\ref{fig_colfan}).

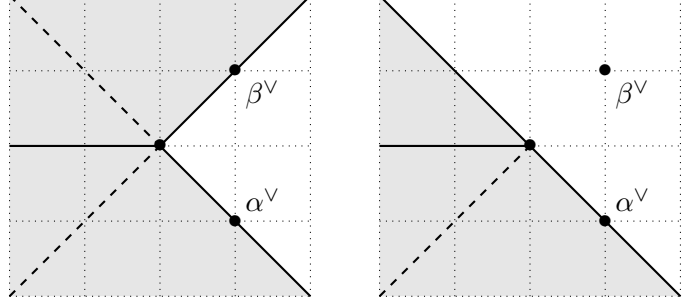
\begin{figure}
\centering
\caption{Colored fans of the quadric and its blow-up}
\label{fig_colfan}
\begin{tikzpicture}  
\fill[color=gray!20] (-2,0) -- (0,0) -- (2,2) -- (-2,2) -- cycle;
\fill[color=gray!20] (-2,0) -- (0,0) -- (2,-2) -- (-2,-2) -- cycle;
\draw [dotted] (-2,-2) grid[xstep=1,ystep=1] (2,2); 
\draw (0,0) node{$\bullet$};
\draw (1,1) node{$\bullet$};
\draw (1,1) node[below right]{$\beta^{\vee}$};
\draw (1,-1) node{$\bullet$};
\draw (1,-1) node[above right]{$\alpha^{\vee}$};
\draw [thick] (-2,0) -- (0,0);
\draw [thick] (2,2) -- (0,0);
\draw [thick] (2,-2) -- (0,0);
\draw [thick, dashed] (-2,2) -- (0,0);
\draw [thick, dashed] (-2,-2) -- (0,0);
\end{tikzpicture}
\qquad 
\begin{tikzpicture}  
\fill[color=gray!20] (-2,0) -- (0,0) -- (-2,2) -- cycle;
\fill[color=gray!20] (-2,0) -- (0,0) -- (2,-2) -- (-2,-2) -- cycle;
\draw [dotted] (-2,-2) grid[xstep=1,ystep=1] (2,2); 
\draw (0,0) node{$\bullet$};
\draw (1,1) node{$\bullet$};
\draw (1,1) node[below right]{$\beta^{\vee}$};
\draw (1,-1) node{$\bullet$};
\draw (1,-1) node[above right]{$\alpha^{\vee}$};
\draw [thick] (-2,0) -- (0,0);
\draw [thick] (-2,2) -- (0,0);
\draw [thick] (2,-2) -- (0,0);
\draw [thick, dashed] (-2,-2) -- (0,0);
\end{tikzpicture}
\end{figure}

\subsection{Combinatorial data associated to divisors}

\subsubsection{Picard group, isotropy characters and polytopes}

The Picard group of $\blQQ{n}{p}$ is generated by: the two colors $D_{\alpha}$ and $D_{\beta}$ whose image under the color map are respectively $\alpha^{\vee}$ and $\beta^{\vee}$, and the $G$-invariant divisors $D_{d}$ and $D_e$, whose image in $(\mathcal{M}\otimes\bbR)^*$ are respectively $-\frac{1}{2}(\alpha^{\vee}+\beta^{\vee})$ and $-\alpha^{\vee}$. 
There are two relations, given by $2(D_{\alpha}-D_e)-D_d=0$ and $D_{\alpha}+D_{\beta}-D_e-D_d$. 
Since we are not interested in integrality issues, we write $\bbR$-divisors as combinations of the two $G$-invariant divisors: 
\[ D(a_d,a_e)=a_dD_d+a_eD_e, \]
and note that any such divisor is \emph{special} and has trivial isotropy character. 
The special polytope and moment polytope thus coincide and live in $\mathcal{M}\otimes \bbR$, where they are given by 
\[ \Delta(a_d,a_e)= \left\{ x \alpha+ y \beta \mid 0\leq x \leq \frac{a_e}{2}, \quad 0\leq y, \quad x+y \leq a_d \right\}. \] 

Knowing the restricted root system (with multiplicities), it is easy to determine the anticanonical divisor, which is here $D(\frac{p+q}{2}-1,p-1)$. 
Note also for later use that the half-sum of positive restricted roots with multiplicities is $(\frac{p}{2}-1)\alpha+(\frac{q}{2}-1)\beta$. 

\begin{figure}
\centering
\caption{Moment polytope for $\blQQ{n}{p}$}
\label{fig_polblQQ}
\begin{tikzpicture}
\newcommand*\multun{4}
\newcommand*\multdeux{6}
\newcommand*\size{.6}
\draw [dotted] (-1*\size,{-(1+\multun/2)*\size}) grid[xstep=\size,ystep=\size] ({(1+\multdeux)*\size},{(1+\multdeux)*\size}); 
\draw (0,0) node{$\bullet$}; 
\draw [very thick, ->] (0,0) -- (\size,-\size) node[below left]{$\alpha$}; 
\draw [very thick, ->] (0,0) -- (\size,\size) node[above left]{$\beta$}; 
\draw [very thick] (0,0) -- ({(\multun+1)*\size/2},{-(\multun+1)*\size/2}) -- ({(\multun+\multdeux+2)*\size/2},{(\multdeux-\multun)*\size/2}) -- ({(\multun+\multdeux+2)*\size/2},{(\multun+\multdeux+2)*\size/2}) -- cycle;
\end{tikzpicture}
\end{figure}
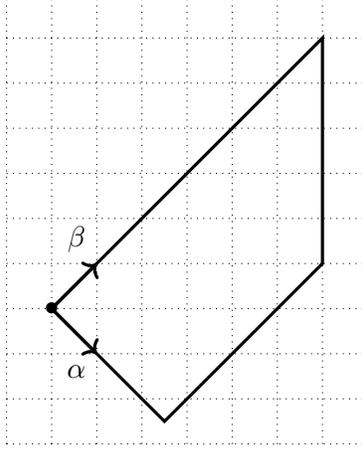

\subsubsection{Duistermaat-Heckman polynomial and barycenters}
Up to a multiplicative constant, the Duistermaat-Heckman polynomial is  
\begin{align*} 
P_{DH}(x\alpha+y\beta) & = \killing{\alpha}{x\alpha+y\beta}^{p-2}\killing{\beta}{x\alpha+y\beta}^{q-2} \\
& = (2x)^{p-2}(2y)^{q-2}
\end{align*}
we set the following notation for the barycenter of moment polytopes with respect to $P_{DH}$ and an arbitrary Lebesgue measure $\dint p$. 
\[ \barDH(a_d,a_e) = \frac{\int_{\Delta(a_d,a_e)}pP_{DH}(p)\dint p}{\int_{\Delta(a_d,a_e)}P_{DH}(p)\dint p} \]

\subsection{Existence of Kähler-Einstein metrics}

The notation $k=p-1$, $l=q-1$ used in the following are introduced for convenience in the proofs. 
With these notations, $p-2=k-1$ is the multiplicity of one of the restricted roots and $q-2=l-1$ is the multiplicity of the other. 
The Duistermaat-Heckman polynomial is thus, up to a constant, $P_{DH}(x\alpha+y\beta)=x^{k-1}y^{l-1}$. 

\subsubsection{Non-existence of Kähler-Einstein metrics}

Set, for $k$, $l$ in $\bbN_{\geq 2}$, 
\begin{equation*} 
I(k,l):= \int_{x=0}^k \int_{y=0}^{k+l-x} (x-(k-1))x^{k-1}y^{l-1} \dint y \dint x. 
\end{equation*}

\begin{lem}
The expression $I(k,l)$ is negative for $k\in \bbN_{>2}$ and $l\in \bbN{\geq 2}$.
\end{lem}

The direct translation in geometrical terms, by \cite[Corollary 1.3]{DH}, justifies the K-unstable members of the family, a major step in proving Theorem~\ref{thm_blQ}:
\begin{cor}
The Fano manifold $\blQQ{n}{p}=\Bl_{Q^{p-2}}(Q^{n-2})$ does not admit a Kähler-Einstein metric if $3<p\leq n-3$. 
\end{cor}

It is a direct consequence from this and the fact that $G$ is semisimple that there cannot exist any Kähler-Ricci soliton or Mabuchi metrics on these manifolds (the soliton or Mabuchi vector field must commute with the action of $G$ hence is zero here). 
Furthermore, the above Theorem has a more precise interpretation: from \cite[p.~653]{DelKSSV} we deduce that $\blQQ{n}{p}$ is not K-semistable, or similarly from \cite[Corollary~1.3]{DH} we deduce that the greatest Ricci lower bound $R(\blQQ{n}{p})$ is strictly less than $1$. 

\begin{proof}
We first derive an explicit formula for the integral. 
Integration in the $y$ variable yields 
\begin{equation*}
I(k,l)= \int_0^k (x^k-(k-1)x^{k-1})\frac{(k+l-x)^l}{l} \dint x
\end{equation*}
Using the change of variables $t=\frac{x}{k}$, we have 
\begin{equation*}
I(k,l)= \frac{k^{k+l+1}}{l} \int_0^1 \left( t^k-\frac{k-1}{k}t^{k-1} \right) \left( 1-t+\frac{l}{k} \right)^l \dint t 
\end{equation*}
By the binomial formula applied to $\left( 1-t+\frac{l}{k} \right)^l$, we write this as a sum 
\begin{equation*}
\frac{k^{k+l+1}}{l} \sum_{j=0}^l {l \choose j} \frac{l^{l-j}}{k^{l-j}} \left( B(k+1,j+1)-\frac{k-1}{k}B(k,j+1) \right)
\end{equation*}
where $B$ denotes the \emph{beta function}, defined by 
\begin{align*}
B(a,b)& =\int_0^1 t^{a-1}(1-t)^{b-1}\dint t \\
& = \frac{(a-1)!(b-1)!}{(a+b-1)!} \quad \text{ at positive integers } a, b
\end{align*}

Let $A_j$ denote the $j$-th summand, so that $lk^{-(k+l+1)}I(k,l)=\sum_{j=0}^lA_j$.
We simplify a bit the expression of the summands:
\begin{align*}
A_j & = {l \choose j} \frac{l^{l-j}}{k^{l-j}} \left( \frac{k!j!}{(k+j+1)!}-\frac{k-1}{k}\frac{(k-1)!j!}{(k+j)!} \right) \\
& = {l \choose j} \frac{l^{l-j}}{k^{l-j+1}} \frac{(k-1)!j!}{(k+j+1)!}\left( k^2-(k-1)(k+j+1) \right) \\
& = {l \choose j} \frac{l^{l-j}}{k^{l-j+1}} \frac{(k-1)!j!}{(k+j+1)!}(j(1-k)+1)
\end{align*}
The major advantage of this expression (over the numerous different ways to compute the integral) is that for $k\geq 3$, all but the first term are negative. 
It thus suffices to compensate the positive first term with some of the negative terms to conclude. 

For most cases, it is enough to consider the first two summands:
\begin{align*}
A_0+A_1 & = \frac{l^l}{k^{l+1}}\frac{1}{k(k+1)} + l\frac{l^{l-1}}{k^{l}}\frac{2-k}{k(k+1)(k+2)} \\
& = \frac{l^l}{k^{l+2}(k+1)(k+2)}\left(-k^2+3k+2 \right)
\end{align*}
The polynomial $-x^2+3x+2$ is negative for $x>\frac{3+\sqrt{17}}{2}$, hence $A_0+A_1$ is negative for $k\geq 4$. 

For the case $k=3$, we consider the additional summand $A_2$. We compute 
\begin{align*}
A_0+A_1+A_2=\frac{l^{l-1}}{36\cdot 3^l}(-5l+9)
\end{align*}
and this is negative for $l>\frac{9}{5}$ hence in particular for $l\geq 2$. 
\end{proof}

\subsubsection{Existence of Kähler-Einstein metrics}

\begin{lem}
We have, for all $l\in \bbN_{\geq 2}$,  
\begin{align*}
J(l):=\int_{x=0}^2\int_{y=0}^{2+l-x} x(y-(l-1))y^{l-1} \dint y \dint x & >0 \\
I(l):= \int_{x=0}^2\int_{y=0}^{2+l-x} x(x-1)y^{l-1} \dint y \dint x & >0 
\end{align*}
\end{lem}

The geometrical translation of the above statement finishes the proof of Theorem~\ref{thm_blQ} in the cases $3\leq p\leq n-3$ by dealing with the case \(p=3\). 

\begin{proof}   
Let us start with the first inequality. We compute \begin{align*}
J(l) & = \int_0^2 x\left( \frac{(2+l-x)^{l+1}}{l+1}-(l-1)\frac{(2+l-x)^{l}}{l} \right) \dint x \\
& = \int_l^{l+2} \frac{t^l}{l(l+1)}\left(-lt^2+(2l^2+2l-1)t-(l+2)(l^2-1) \right) \dint t \\
& = \frac{(l+2)^{l+1}(4l^2+7l-2)-l^{l+1}(-4l^2+l+12)}{l(l+1)(l+2)(l+3)} 
\end{align*}
Note that, for $l\geq 2$, $4l^2+7l-2$ is positive, and $(l+2)^{l+1}>l^{l+1}$, so we have 
\begin{align*}
l^{-l}(l+1)(l+2)(l+3)J(l) & >(4l^2+7l-2)-(-4l^2+l+12) \\
& > 8l^2+6l-14.
\end{align*}
It is easy to check that the latter polynomial is positive for $l>1$ hence \emph{a~fortiori} for $l\geq 2$.

For the second inequality, we compute using the same method 
\begin{align*}
I(l) & = \frac{(l+2)^{l+2}-(7l+12)l^{l+1}}{(l+2)(l+3)} \\
& = \frac{(l+2)^{l+2}}{(l+2)(l+3)}\left( 1-\left( 1-\frac{2}{l+2} \right)^{l+2}\frac{7l+12}{l} \right)
\end{align*}
By a classical inequality, $\left(1-\frac{2}{l+2} \right)^{l+2}<e^{-2}$, hence it is enough to have 
\begin{equation*} 
7l+12<le^2.
\end{equation*}
The inequality above is satisfied as long as $l\geq 31$ and the statement is proved in these cases. 
It remains to check a finite number of cases (for $2\leq l\leq 30$). 
We omit the details as it is tedious and without difficulty. 
\end{proof}

\section{Remaining cases of blow-ups of quadrics along a linear subquadric}
\label{sec_BlQred}

\subsection{Description of the manifolds}

In this last section, we will be interested in the blow-ups of $Q^{n-2}$ along $Q^{n-4}$, the blow-up along $Q^0$, and the blow up along a single point (which is one half of $Q^0$). In particular, in the notations of the previous section, we deal with the cases \(p=2\) and \(p=n-2\).  

We hence consider the case when the group acting on the quadric is $G=\SO_{n-2}\times\SO_2$ (and we assume $5\leq n$, the case of the blow up of $\bbP^1\times \bbP^1$ at one or two points being well-known). 
Since $\SO_2 = \bbC^*\), the group $G$ is not semisimple. 
Furthermore, there are more orbits in the quadric:
the zero-dimensional quadric $Q^0$ consists of \emph{two} orbits of $G$, 
there is still the closed orbit $Q^{n-4}$, 
there are \emph{two} codimension one orbits, formed by points whose projections belong to $Q^{n-4}$ and one point in $Q^0$, 
and finally the complement of all these lower dimensional orbits is the open dense orbit.

\subsection{Combinatorial data associated to the manifolds}

\subsubsection{The open orbit}

The open dense orbit is a (reductive) symmetric space of type $A_1$ and rank two. 
The simple root is denoted, consistently with the previous section, by $\alpha$.
Its multiplicity is $p-2=n-4$. 
The spherical lattice $\mathcal{M}$ is generated by some element $\varpi \in \mathfrak{X}(T/(T\cap [G,G]))$ and $\frac{\varpi+\alpha}{2}$ 
The valuation cone $\mathcal{V}$ is still the negative restricted Weyl chamber, defined here by $\alpha\leq 0$ in $(\mathcal{M}\otimes\bbR)^*$.
There is a single color $D_{\alpha}$ which is sent to $\alpha^{\vee}$ by the color map. 
Let $f$ be the element of $(\mathcal{M}\otimes\bbR)^*$ defined by $\alpha(f)=0$ and $\varpi(f)=2$, so that $f$ and $\frac{f+\alpha^{\vee}}{2}$ generate the lattice dual to $\mathcal{M}$. 

\subsubsection{Colored fans and embeddings}
The colored fan of the quadric is formed by the following cones:
\begin{itemize}
\item $(\{0\},\emptyset)$ corresponding to the open dense orbit,
\item the ray $(\bbR_+ (f-\alpha^{\vee}),\emptyset)$ corresponding to a codimension one orbit whose closure we denote by $D_+$,
\item the ray $(-\bbR_+ (f+\alpha^{\vee}),\emptyset)$ corresponding to a codimension one orbit whose closure we denote by $D_-$,
\item the colorless cone $(\Cone(f-\alpha^{\vee},-f-\alpha^{\vee}),\emptyset)$ corresponding to $Q^{n-4}$
\item the colored cone $(\Cone(\alpha^{\vee},f-\alpha^{\vee}),\{\alpha^{\vee}\})$ corresponding to one point in $Q^0$,
\item and the colored cone $(\Cone(\alpha^{\vee},-f-\alpha^{\vee}),\{\alpha^{\vee}\})$ corresponding to the other point in $Q^0$.
\end{itemize}
To obtain the colored fans of the blow-ups, it suffices to make the obvious modifications while adding:
\begin{itemize}
\item the ray $(\pm\bbR_+ f,\emptyset)$ for the blow-up at one point of $Q^0$, or both for the blow-up along the full $Q^0$, we denote by $E_{\pm}$ the corresponding divisors,
\item the ray $(-\bbR_+ \alpha^{\vee},\emptyset)$ for the blow-up at $Q^{n-4}$, we denote by $E$ the corresponding divisor. 
\end{itemize}

The images of the above divisors in $(\mathcal{M}\otimes \bbR)^*$ are as illustrated in Figure~\ref{fig_colorsQ}.
\begin{figure}
\centering
\caption{Images of divisors in $(\mathcal{M}\otimes \bbR)^*$}
\label{fig_colorsQ}
\begin{tikzpicture}  
\draw [dotted] (-2,-2) grid[xstep=1,ystep=1] (2,2); 
\draw [thick, ->] (0,0) -- (1,-1);
\draw (.7,-.7) node[left]{$\alpha^{\vee}$};
\draw [thick, ->] (0,0) -- (1,1);
\draw (.5,.5) node[right]{$f$};
\draw (1,1) node{$\bullet$};
\draw (1,1) node[above right]{$E_+$};
\draw (-1,-1) node{$\bullet$};
\draw (-1,-1) node[below left]{$E_-$};
\draw (0,1) node{$\bullet$};
\draw (0,1) node[above]{$D_+$};
\draw (-1,0) node{$\bullet$};
\draw (-1,0) node[left]{$D_-$};
\draw (1,-1) node{$\bullet$};
\draw (1,-1) node[below right]{$D_{\alpha}$};
\draw (-1,1) node{$\bullet$};
\draw (-1,1) node[above left]{$E$};
\end{tikzpicture}
\end{figure}
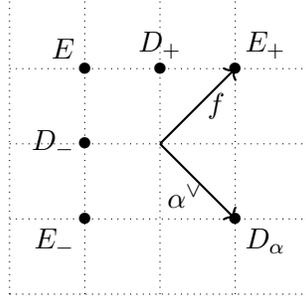

\subsection{Combinatorial data associated to divisors}

\subsubsection{Picard group, isotropy characters and polytopes:}
The Picard group of the blow-up of $Q^{n-2}$ along $Q^{n-4}$ is generated by $D_+$, $D_-$, $D_{\alpha}$, and $E$, with the relations $D_-=D_+$ and $D_++D_-+2E-2D_{\alpha}=0$. 
It is convenient to represent the $\bbR$-divisors as combinations 
\[ D_e(a,a_e)=a(D_++D_-)+a_eE \]
with $a$, $a_e\in \bbR$. 
They are special divisors, with trivial isotropy characters. 
The corresponding special (or moment) polytope is 
\[ \Delta_e(a,a_e)=\{ x\alpha+y\varpi \mid 0\leq x\leq \frac{a_e}{2}, \quad x-a\leq y\leq a-x  \}. \]
The divisor $D_e(a,a_e)$ is ample if and only if $0<a_e<2a$.
The anticanonical divisor is $D_e(\frac{n}{2}-1,n-3)$. 
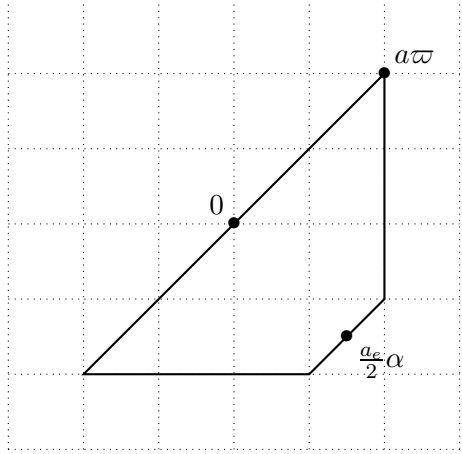
\begin{figure}
\centering
\caption{The polytope $\Delta_e(a,a_e)$}
\label{fig_polQn4}
\begin{tikzpicture}  
\draw [dotted] (-3,-3) grid[xstep=1,ystep=1] (3,3);
\draw (0,0) node{$\bullet$};
\draw (0,0) node[above left]{$0$};
\draw [thick] (-2,-2) -- (1,-2) -- (2,-1) -- (2,2) -- cycle;
\draw (1.5,-1.5) node{$\bullet$};
\draw (1.5,-1.5) node[below right]{$\frac{a_e}{2}\alpha$};
\draw (2,2) node{$\bullet$};
\draw (2,2) node[above right]{$a\varpi$};
\end{tikzpicture}
\end{figure}

We now consider the blow-up of $Q^{n-2}$ along one point of $Q^0$, say the one whose exceptional divisor we denoted by $E_+$. 
The relations in the Picard group are $2E_++D_+-D_-=0$ and $2D_{\alpha_1}-D_+-D_-=0$.
We may then write any $\bbR$-divisor as combinations 
\[ D_+(a,a_+)=a(D_++D_-)+a_+E_+ \]
with $a$, $a_+\in \bbR$. 
They are special divisors, with trivial isotropy characters. 
The corresponding special (or moment) polytope is 
\[ \Delta_+(a,a_+)=\left\{ x\alpha+y\varpi \mid 0\leq x, \quad \frac{-a_+}{2} \leq y, \quad x-a\leq y\leq a-x  \right\}. \]
The divisor $D_+(a,a_+)$ is ample if and only if $0<a_+<2a$.
The anticanonical divisor is $D_+(\frac{n}{2}-1,1)$.
\begin{figure}
\centering
\caption{The polytope $\Delta_+(a,a_+)$}
\label{fig_polBlQpt}
\begin{tikzpicture}  
\draw [dotted] (-3,-3) grid[xstep=1,ystep=1] (3,3);
\draw (0,0) node{$\bullet$};
\draw (0,0) node[above left]{$0$};
\draw [thick] (-1,-1) -- (0,-2) -- (2,-2) -- (2,2) -- cycle;
\draw (2,-2) node{$\bullet$};
\draw (2,-2) node[below right]{$a\alpha$};
\draw (-1,-1) node{$\bullet$};
\draw (-1,-1) node[below left]{$-\frac{a_+}{2}\varpi$};
\end{tikzpicture}
\end{figure}
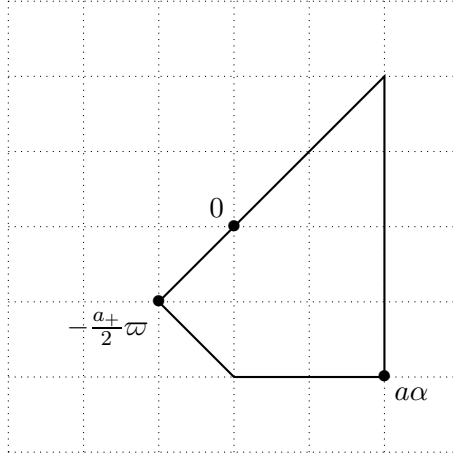

We now consider the blow-up of $Q^{n-2}$ along both points of $Q^0$.
The relations are now $2(E_+-E_-)+D_+-D_-=0$ and $2D_{\alpha_1}-(D_++D_-)=0$.
We may then write any $\bbR$-divisor as combinations 
\[ D_{\pm}(a,a_+,a_-)=a(D_++D_-)+a_+E_+ +a_-E_- \]
with $a$, $a_+$, $a_-\in \bbR$. 
They are special divisors, with trivial isotropy characters. 
The corresponding special (or moment) polytope is 
\[ \Delta_{\pm}(a,a_+,a_-)=\left\{ x\alpha+y\varpi \mid 0\leq x, \quad \frac{-a_+}{2} \leq y \leq \frac{a_-}{2}, \quad x-a\leq y\leq a-x  \right\}. \]
The divisor $D_{\pm}(a,a_+,a_-)$ is ample if and only if $0<a_+<2a$ and $0<a_-<2a$.
The anticanonical divisor is $D_{\pm}(\frac{n}{2}-1,1,1)$.
\begin{figure}
\centering
\caption{The polytope $\Delta_{\pm}(a,a_+,a_-)$}
\label{fig_polBlQ2pt}
\begin{tikzpicture}  
\draw [dotted] (-3,-3) grid[xstep=1,ystep=1] (3,3);
\draw (0,0) node{$\bullet$};
\draw (0,0) node[above left]{$0$};
\draw [thick] (-1,-1) -- (0,-2) -- (2,-2) -- (2,1) -- (1.5,1.5) -- cycle;
\draw (2,-2) node{$\bullet$};
\draw (2,-2) node[below right]{$a\alpha$};
\draw (-1,-1) node{$\bullet$};
\draw (-1,-1) node[below left]{$-\frac{a_+}{2}\varpi$};
\draw (1.5,1.5) node{$\bullet$};
\draw (1.5,1.5) node[above right]{$\frac{a_-}{2}\varpi$};
\end{tikzpicture}
\end{figure}
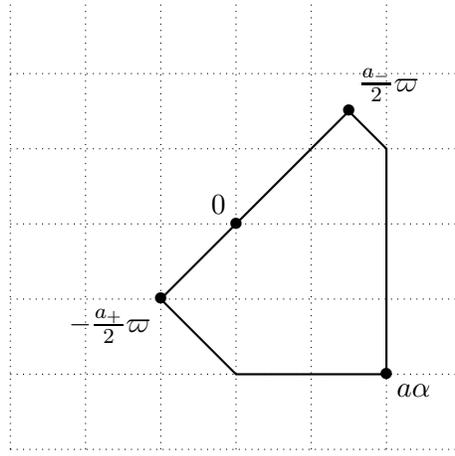

\subsubsection{Duistermaat-Heckman polynomial and barycenters}
Up to a multiplicative constant, the Duistermaat-Heckman polynomial is  
\begin{align*} 
P_{DH}(x\alpha+y\varpi) & = \killing{\alpha}{x\alpha+y\varpi}^{n-4} \\
& = (2x)^{n-4}
\end{align*}
The half sum of restricted roots is $(\frac{n}{2}-2)\alpha_1$

We set the following notations for the barycenters of the several blow-ups of $Q^{n-2}$:
\begin{align*}
\barDH_e(a,a_e)& =\frac{\int_{\Delta_e(a,a_e)}pP_{DH}(p)\dint p}{\int_{\Delta_e(a,a_e)}P_{DH}(p)\dint p} \\
\barDH_+(a,a_+)& =\frac{\int_{\Delta_+(a,a_+)}pP_{DH}(p)\dint p}{\int_{\Delta_+(a,a_+)}P_{DH}(p)\dint p} \\
\barDH_{\pm}(a,a_+,a_-)& =\frac{\int_{\Delta_{\pm}(a,a_+,a_-)}pP_{DH}(p)\dint p}{\int_{\Delta_{\pm}(a,a_+,a_-)}P_{DH}(p)\dint p} \\
\end{align*}

\subsection{Existence of Kähler-Einstein metrics for blow-ups along a full quadric}

We will now finish the proof of Theorem~\ref{thm_blQ}.

\subsubsection{Case of $\Bl_{Q^{n-4}}Q^{n-2}$}

Set 
\begin{align*}
I_V & = \int_{x=0}^{n-3}\int_{y=x-n+2}^{n-2-x}x^{n-4} \dint y \dint x \\
I_{\alpha} & = \int_{x=0}^{n-3}\int_{y=x-n+2}^{n-2-x}x^{n-3} \dint y \dint x \\
I_{\varpi} & = \int_{x=0}^{n-3}\int_{y=x-n+2}^{n-2-x}yx^{n-4} \dint y \dint x 
\end{align*}
then $\Bl_{Q^{n-4}}(Q^{n-2})$ is Kähler-Einstein if and only if 
\[ I_{\varpi}=0 \qquad \text{and} \qquad \frac{I_{\alpha}}{I_V}>n-4. \] 
The vanishing $I_{\varpi}=0$ is obvious by symmetry with respect to the reflection leaving $\bbR\varpi$ invariant, while a straightforward computation yields 
\[ \frac{I_{\alpha}}{I_V}= \frac{2(n-3)^2(n-2)}{(n-1)(2n-5)}. \]
Then $\frac{I_{\alpha}}{I_V}>n-4$ is equivalent to $7n^2-11n+2>0$ which is true for $n>\frac{11+\sqrt{65}}{14}$ hence certainly for all $n\geq 5$.
So we have proved that $\Bl_{Q^{n-4}}Q^{n-2}$ is Kähler-Einstein for all $n\geq 5$.

\subsubsection{Case of $\Bl_{Q^{0}}Q^{n-2}$}

Here there is an obvious symmetry of the polytope yielding the vanishing of the Futaki character, hence no obstruction in the direction of $\varpi$. 
It remains to check the inequality
\[ \frac{I_{\alpha}}{I_{V}}-(n-4)>0 \] 
where 
\begin{align*}
I_V & = 2\int_{y=0}^{1}\int_{x=0}^{n-2-y}x^{n-4} \dint x \dint y \\
I_{\alpha} & = 2\int_{y=0}^{1}\int_{x=0}^{n-2-y}x^{n-3} \dint x \dint y. 
\end{align*}
We have, by computation of the integrals,  
\begin{align*} 
\frac{I_{\alpha}}{I_{V}}& = \frac{(n-3)\left((n-2)^{n-1}-(n-3)^{n-1}\right)}{(n-1)\left((n-2)^{n-2}-(n-3)^{n-2}\right)} \\
\end{align*}

Using the fact that 
\[ \left((n-2)^{n-1}-(n-3)^{n-1}\right)>(n-2)\left((n-2)^{n-2}-(n-3)^{n-2}\right)\] 
the previous inequality is implied by 
\[ (n-3)(n-2)-(n-4)(n-1)>0 \]
which is true (recall that we can assume $n\geq 5$ throughout). 
We have proved that $\Bl_{Q^{0}}Q^{n-2}$ is Kähler-Einstein.

\subsection{There are no Mabuchi metrics on $\Bl_{\text{pt}}Q^{n-2}$}

We finally consider the blow-up of $Q^{n-2}$ at a point. 
There is an obvious lack of symmetry in the polytope (even with respect to the Duistermaat-Heckman measure) which ensures the the Futaki character will not vanish. 
Hence there are no Kähler-Einstein metrics on $\Bl_{\text{pt}}Q^{n-2}$. 
An even shorter proof is that the automorphism group is not reductive. 
We prove in fact a stronger result: 
\begin{thm}
For $n\geq 5$, the blow-up of $Q^{n-2}$ at one point does not admit Mabuchi metrics. 
\end{thm}

\begin{proof}
By \cite[Corollary 1.4]{DH} and the combinatorial description above, there exists a Mabuchi metric on $\Bl_{\text{pt}}Q^{n-2}$ if and only if we can find a linear function $y\mapsto Ay+B$ with $0\neq A$, $B\in \bbR$, such that the following three conditions are satisfied:
\[ \left(\int_{y=\frac{-1}{2}}^{0}\int_{x=0}^{\frac{n-2}{2}+y} +\int_{y=0}^{\frac{n-2}{2}}\int_{x=0}^{\frac{n-2}{2}-y}\right) y(Ay+B)x^{n-4} \dint x \dint y  =0, \]
\[ \left(\int_{y=\frac{-1}{2}}^{0}\int_{x=0}^{\frac{n-2}{2}+y} +\int_{y=0}^{\frac{n-2}{2}}\int_{x=0}^{\frac{n-2}{2}-y}\right) (x-\frac{n-4}{2})(Ay+B)x^{n-4} \dint x \dint y =0, \]
and 
\[ \frac{-B}{A} < \frac{-1}{2} \quad \text{ or } \quad \frac{-B}{A} >\frac{n-2}{2}. \]
We will show that when the first condition is satisfied, the third condition is violated. 
Assume $A$ and $B$ are such that the first condition is satisfied. 
Then $\frac{-B}{A}$ is equal to 
\[ \frac{\left(\int_{y=\frac{-1}{2}}^{0}\int_{x=0}^{\frac{n-2}{2}+y} +\int_{y=0}^{\frac{n-2}{2}}\int_{x=0}^{\frac{n-2}{2}-y}\right) y^2x^{n-4} \dint x \dint y}{\left(\int_{y=\frac{-1}{2}}^{0}\int_{x=0}^{\frac{n-2}{2}+y} +\int_{y=0}^{\frac{n-2}{2}}\int_{x=0}^{\frac{n-2}{2}-y}\right) yx^{n-4} \dint x \dint y} \] 
The numerator is obviously positive, and the denominator is positive thanks to the lack of symmetry, and comparison with the Kähler-Einstein case of $\Bl_{Q^{0}}Q^{n-2}$. 
In order to show that there are no Mabuchi metrics, it thus suffices to show that 
\[ \left(\int_{y=\frac{-1}{2}}^{0}\int_{x=0}^{\frac{n-2}{2}+y} +\int_{y=0}^{\frac{n-2}{2}}\int_{x=0}^{\frac{n-2}{2}-y}\right)
y(y-\frac{n-2}{2})x^{n-4} \dint x \dint y \leq 0. \]
Integrating, then multiplying by $2^n n(n-1)(n-3)$, this is equivalent to 
\[ 4(n-2)^{n-1}-(n-3)^{n-2}(2n^2+n-9) \leq 0. \]
Since 
\[ \left(\frac{n-2}{n-3}\right)^{n-3}=\left(1-\frac{1}{n-3}\right)^{n-3}\leq e^{-1}, \]
the inequality is implied by 
\[ 0\leq 2n^3-9n^2+4n+11 \]
which certainly holds for $n\geq 5$. 
We have shown that $\Bl_{\text{pt}}Q^{n-2}$ admits no Mabuchi metrics (for $n\geq 5$).
\end{proof}

\bibliographystyle{plain}
\bibliography{infinite}

\end{document}